\documentclass[a4paper,12pt]{article}

\usepackage[margin=3cm,footskip=1cm]{geometry}

\usepackage[T1]{fontenc}

\usepackage{amsmath,amssymb,amsthm,bm,mathtools,xspace,booktabs,url}
\usepackage{graphicx,etoolbox}

\usepackage[shortlabels]{enumitem}
\setlist{
  listparindent=\parindent,
  parsep=0pt,
}

\usepackage[numbers,sort&compress]{natbib}			

\AtBeginEnvironment{thebibliography}{%
  \small
  \setlength\itemsep{0.75em plus 0.5em minus 0.75em}%
}

\usepackage{caption}
\usepackage[raggedright]{titlesec}

\numberwithin{equation}{section} 

\theoremstyle{plain} 
\newtheorem{theorem}{Theorem}[section]
\newtheorem{Lemma}[theorem]{Lemma}
\newtheorem{Proposition}[theorem]{Proposition}
\newtheorem{Corollary}[theorem]{Corollary}

\theoremstyle{definition} 

\newtheorem{Remark}[theorem]{Remark}
\newtheorem{Algorithm}[theorem]{Algorithm}

\usepackage{authblk}

\makeatletter
\newcommand\CorrespondingAuthor[1]{%
  \begingroup%
  \def\@makefnmark{}%
  \footnotetext{Corresponding author: #1}%
  \endgroup%
}
\makeatother

\makeatletter
\renewenvironment{abstract}{%
  \small%
  \providecommand\keywords{%
    \par\medskip\noindent\textit{Keywords:}\xspace}%
  \begin{center}%
    \bfseries \abstractname\vspace{-.5em}\vspace{\z@}%
  \end{center}%
  \quote%
}
{
\endquote}
\makeatother

\usepackage[above,below,section]{placeins}

\usepackage[load-configurations=abbreviations]{siunitx}
\usepackage[all]{onlyamsmath}

\usepackage{lipsum}

\DeclareMathOperator\Var{Var}

\usepackage{xcolor}

\usepackage{mathrsfs}
\DeclareMathAlphabet{\mathpzc}{T1}{pzc}{m}{it}

\newcommand{\LW}{\mathpzc{W}}

\newcommand{\nat}{{\mathbb N}}
\newcommand{\Laplace}[1]{\mathcal{L}#1\,}
\newcommand{\e}{\mathrm{e}}
\newcommand{\dd}{\mathrm{d}}

\newcommand{\oh}{{\mathit{o}}}
\newcommand{\Oh}{{\mathit{O}}}
\newcommand{\Exp}{\mathbb{E}}
\newcommand{\dNormal}{\sim\mathrm{N}\,}
\newcommand{\dLNormal}{\sim\mathrm{LN}\,}
\newcommand{\Prob}{\mathbb{P}}
\newcommand{\Ind}{\mathbb I}

\def\*{\discretionary{}{\hbox{\ensuremath\cdot}\thinspace}{}}

\begin{document}

\title{Exponential Family Techniques\\ for the Lognormal Left Tail
%
}

\date{}

\author[1]{S\o ren Asmussen} \author[1]{Jens Ledet Jensen}
\author[2]{Leonardo Rojas-Nandayapa} \affil[1]{Department of
  Mathematics, Aarhus University} \affil[2]{School of Mathematics and
  Physics, University of Queensland} 

\maketitle


\begin{abstract} 
Let $X$ be lognormal$(\mu,\sigma^2)$ with density $f(x)$, let
$\theta>0$ and define $\Laplace(\theta)=\Exp\e^{-\theta X}$.
We study properties of the exponentially tilted density (Esscher transform) $f_\theta(x)
=\e^{-\theta x}f(x)/\Laplace(\theta)$, in particular its moments, its asymptotic
form as $\theta\to\infty$ and asymptotics for the Cram\'er function;
the asymptotic
formulas involve the Lambert W function.
This is used to provide two different numerical methods for 
evaluating the left tail 
probability of lognormal sum $S_n=X_1+\cdots+X_n$: a saddlepoint approximation and an exponential twisting 
importance sampling estimator. For the latter we demonstrate the asymptotic consistency 
by proving logarithmic efficiency in terms of the mean square error. 
Numerical examples for the c.d.f.\ $F_n(x)$ and the p.d.f.\ $f_n(x)$ of $S_n$ are given  in a range of values of $\sigma^2,n,x$ motivated from portfolio Value-at-Risk calculations.
 
\end{abstract}

\section{Introduction}

The lognormal distribution arises in a wide variety of disciplines such 
as engineering, economics, insurance or finance, and is often employed 
in modeling across the sciences \cite{AitchisonBrown1957,CrowShimizu1988,
JohnsonKotzBalakrishnan94,LimpertStahelAbbt2001,Dufresne2009}.
In consequence, it is natural that sums of lognormals come up in a 
number of contexts. For instance, a basic example in finance is the Black-Scholes 
model, which asserts that security prices can be modeled as independent 
lognormals (equivalently, the logprices are 
independent normally distributed). This implies that the value of a portfolio with $n$ 
securities can be conveniently modeled as a sum of lognormals. 
Another example occurs in the valuation of arithmetic Asian options 
where the payoff depends on the finite sum of correlated lognormals
\cite{MilevskyPosner1998,Dufresne2004}. In 
insurance, individual claim sizes are often modeled as independent 
lognormals, so the total claim amount after certain period is a random 
sum of lognormals \cite{ThorinWikstad1977}.  A further example occurs in telecommunications, where
the inverse of the signal-to-noise ratio (a measure of performance in wireless systems)
can be modeled as a sum of i.i.d.\ lognormals \cite{Gubner2006}.

However, the distribution of a sum of $n$ lognormals $S_n$ is not available 
in explicit form
and its numerical approximation is considered to be a challenging problem.
In consequence, a number of methods for its evaluation has been developed
across several decades, but these can rarely deliver  
arbitrary precisions in the whole support of the distribution, particularly in the tails.
The later case is of key relevance in certain applications which often require to evaluate 
tail probabilities at very high precisions. 
For instance, the Value-at-Risk (VaR) is an important measure of 
market risk  defined as an appropriate $(1-\alpha)$ quantile of the loss distribution,
and the standard financial treatise Basel\,II~\cite{BaselII}
asks for calculations of the VaR for so 
small values as $\alpha=0.03$\%.

When considering lognormals sums, the literature has sofar concentrated on the right tail
(with the exception of the recent paper  \cite{GT2013} by Gulisashvili \& Tankov).
In this paper, our object of study is rather the left tail and certain mathematical problems
that naturally come up in this context.
To be precise, let $Y_i$ be normal$(\mu_i,\sigma^2_i)$
(we don't at the moment specify the dependence structure), let $X_i=\e^{Y_i}$
and $S_n=X_1+\cdots+X_n$. We then want to compute
$\Prob(S_n\le z)$ in situations where this probability is small.

An obvious motivation for this problem comes from the VaR problem.
Here $S_n$ may represent the future value of the portfolio. If $\Pi$ is the
present value, $\Pi-S_n$ is then the loss, and so calculation of
$\alpha$-quantiles are equivalent to left tail calculations for $S_n$.
A further example occurs in the wireless systems setting, where
an outage occurs when the signal-to-noise ratio exceeds a
large threshold. The outage probability is therefore obviously related to the left tail probability of a lognormal sum. 

The problem of approximating the distribution of a sum of i.i.d.\ lognormals
has as mentioned a long history. 
The classical approach is to approximate the distribution of a sum of i.i.d.\ lognormals
with another lognormal distribution.  This goes back at least to Fenton~\cite{Fenton1960}
in 1960 and it is nowadays 
known as the \emph{Fenton-Wilkinson method};
according to Marlow~\cite{Marlow1967} this approximation was already 
used by Wilkinson since 1934. 
However, the Fenton-Wilkinson method, being a central
limit type result, can deliver rather inaccurate approximations of the distribution
of the lognormal sum when the number of summand is rather small or the dispersion 
parameter is too high---in particular in the 
tail regions.  Another topic which has been much studied recently 
is approximations and simulation algorithms for right tail probabilities $\Prob(S_n\ge y)$
under heavy-tailed assumptions and allowing for dependence, see
in particular \cite{AsmussenRojas08,FossRichards2010,MitraResnick2008,AsmussenBlanchetJunejaRojasNandayapa2008,BlanchetRojasNandayapa2011}. For further literature surveys, see \cite{GT2013}.

Our approach  is to use the saddlepoint approximations and a
closely related simulation algorithm based on the same exponential
change of measure. 
 This 
 requires i.i.d. assumptions,
in particular  $\mu_i\equiv\mu$, $\sigma^2_i\equiv\sigma^2$.
Since $\mu$ is just a scaling factor, we will assume $\mu=0$.
The saddlepoint approximation occurs in various (closely related) forms,
but all involve the function $\kappa(\theta)\ =\ \log \Laplace(\theta)$ where
\[
 \Laplace(\theta)\ =\ \Exp \e^{-\theta X_i} \ = \ 
\int_0^\infty \e^{-\theta x}f(x)\,\dd x
\text{\ \ with\ \ }f(x)\ =\ \frac{1}{x\sigma\sqrt{2\pi}}\e^{-\log^2x/2\sigma^2}\]
and its two first derivatives $\kappa'(\theta)$, $\kappa''(\theta)$
[note that since the right tail of the lognormal distribution is heavy,
these quantities are only defined for $\theta\ge 0$].
Define the exponentially
tilted density $f_{\theta}(x)$ (Esscher transform) by 
$f_{\theta}(x)=\e^{-\theta x-\kappa(\theta)}f_\sigma(x)$, and let
its corresponding c.d.f./probability distribution be $F_\theta$
with expectation operator $\Exp_\theta$. Then
\begin{equation}\label{10.2b}
\kappa'(\theta)=-\Exp_\theta X\,,\ \ \ \kappa''(\theta)=\Var_\theta X
\end{equation}
and can  connect the given distribution 
of $S_n$ (corresponding to $\theta=0$) to the $\Prob_\theta$-distribution
by means of the likelihood ratio identity
\[\Prob(S_n\in A)\ =\ \Exp_\theta\bigl[\exp\{\theta S_n+n\kappa(\theta)\};\,
S_n\in A\bigr]\,.\] 
The details of the saddlepoint approximation involve
writing $z=nx$, defining
the \emph{saddlepoint} or \emph{Cram\'er function} $\theta(x)$ as the solution of the equation
$\kappa'\bigl(\theta(x)\bigr)=-x$ and taking $\theta=\theta(x)$.
This choice of $\theta$ means that $\Exp_\theta S_n=z$ so that
the $\Prob_\theta$-distribution is centered around $z$ and central
limit expansions apply. For a short exposition of the implementation
of this program in its simplest form, see \cite[p.\,355]{AsmussenAPQ2003}.

The application of  saddlepoint approximations to the lognormal left tail
appears first to have appeared in the third author's 2008 Dissertation \cite{PhDThesis},
but in a more incomplete and preliminary form than the one presented here.
A first difficulty is that $\kappa(\theta)$ is not explicitly available for the
lognormal distribution. However, approximations with error rates were recently given
in the companion paper  \cite{AJRN2014}. 
The result is in terms of
the Lambert W function $\LW(a)$ \cite{CorlessGonnetHareJeffreyKnuth1996}, defined as the unique solution
of $\LW(a)\e^{\LW(a)}=a$ for $a>0$. The expression for  $\kappa(\theta)$
in \cite{AJRN2014} is
\begin{equation}\label{10.2a}
\Laplace(\theta)\ =\ 
\frac{\exp\bigg\{-\dfrac{\LW^2(\theta\sigma^2)+2\,
    \LW(\theta\sigma^2)}{2\sigma^2}\bigg\}}
  {\sqrt{1+\LW(\theta\sigma^2)}}
  \Exp\big[g_0(\sigma_\theta Z)\big],
\end{equation}
where $Z\sim\mathrm{N}(0,1)$, $\sigma^2_\theta=\sigma^2/(1+\LW(\theta\sigma^2))$ and $g_0$ is a certain function such that $\Exp\big[g_0(\sigma_\theta Z)\big]$ is close to 1 (see Section~\ref{S:ExpFam} for more detail;
we also give  an extension to expectations of the form $\Exp X \e^{-\theta X}$ there).

The paper is organized as follows.
In Section~\ref{S:Laplace}, we study the exponential family
$(F_\theta)_{\theta\ge 0}$. We give a heuristic  proof that $F_\theta$ can be 
approximated by a lognormal distribution and obtain an approximation
of the Legendre-Fenchel transform of the lognormal distribution. 
The first important application of our results, namely the 
saddlepoint approximation for  $\Prob(S_n\le x)$, is given in Section~\ref{S:Saddle}. 
The second is a Monte Carlo estimator for  $\Prob(S_n\le x)$
given in Section~\ref{SS:IS}. It follows a classical model 
(Asmussen \& Glynn~\cite[VI.2]{AsmussenGlynn2007})
by attempting importance sampling with
importance distribution is $F_{\theta(x)}$,
but  the implementation faces the difficulty that neither $\theta(x) $ nor $\kappa\bigl(\theta(x)\bigr)$ are explicit, but must be approximated (various approaches to overcome
this are discussed in the numerical examples presented
in Section~\ref{S:Numerical}). The algorithm requires simulation
from $F_\theta$ for certain $\theta$, and we 
suggest an acceptance-rejection (A-R) for this with a certain Gamma proposal;
the analysis gives as byproduct that this Gamma is an excellent
approximation of $F_\theta$.
The Appendix contains various supplements, in particular
a proof that the importance sampling proposed in Section~\ref{SS:IS}
has a certain asymptotical efficiency property.

\section{The exponential family generated by the lognormal distribution}\label{S:ExpFam}
\label{S:Laplace}


We let $F$  be the cumulative distribution function of $X$ and adopt
the notation $X\dLNormal(0,\sigma^2)$.   
For convenience, 
we write $f_{n}$ and $F_{n}$ 
for the pdf and cdf of $S_n$, respectively. 

The exponential tilting scheme in the Introduction is often also 
referred to as \emph{Esscher transformation}. Note that since $\kappa(\theta)$
is well-defined for all $\theta>0$, one avoids for $x<\Exp X$
(the relevant case for our left tail problem) the difficulties in large deviations
theory associated with boundary problems when defining the saddlepoint
$\theta(x)$ and which lead into minimizing the convex conjugate $\kappa^\ast(x)=\kappa(\theta(x))+x\theta(x)$
(also called the \emph{Legendre-Fenchel transform}).

%

\begin{Proposition}\label{Th1}
 Let $X\dLNormal(0,\sigma^2)$, $k\in\nat^+$, $\theta>0$. Then
 \begin{align}
  \Exp[X^k\e^{-\theta X}]
	&=\Exp\big[g_k(\sigma_{k,\theta}Z)\big]\,
\frac{\exp\bigg\{-\dfrac{\LW^2(\theta\sigma^2\e^{k\sigma^2})+2\,
    \LW(\theta\sigma^2\e^{k\sigma^2})-k^2\sigma^4}{2\sigma^2}\bigg\}}
  {\sqrt{1+\LW(\theta\sigma^2\e^{k\sigma^2})}}.\label{Th1eq}
\intertext{In addition,  also }  
\Exp[X^k\e^{-\theta X}]   
  &=\Exp\big[g_0\big((Z+k)\sigma^2_{0,\theta}\big)\big]\nonumber\\
    &\qquad \times	\frac{\exp\bigg\{-\dfrac{\LW^2(\theta\sigma^2)
     +2\,\LW(\theta\sigma^2)}{2\sigma^2}+k\big(\LW(\theta\sigma^2)+
     {k\sigma_{0,\theta}^2}/{2}\big)  \bigg\}}{\sqrt{\,\LW(\theta\sigma^2)+1}},\label{Th2eq}
 \end{align}
 where $Z\dNormal(0,1)$ and
 \begin{equation*}
\sigma^2_{k,\theta}=\frac{\sigma^2} {1+\LW(\theta\sigma^2\e^{k\sigma^2})},\quad
g_k(w)=\exp\bigg\{-\frac{\LW(\theta\sigma^2\e^{k\sigma^2})}{\sigma^2}
\big(\e^{w}-1-w-w^2/2\big)\bigg\}.
 \end{equation*}
 Here
 \begin{equation}\label{10.2c}
  \Exp[g_k(\sigma_{k,\theta}Z)]=1+\Oh\big((\log\theta)^{-1}\big),\qquad
    \Exp[g_0(\sigma_{k,\theta}Z)]=1 +\Oh\big((\log\theta)^{-2}\big).
 \end{equation}
\end{Proposition}

The proof of Proposition \ref{Th1} follows the same lines
as the proof of Proposition 2.1 in the companion paper \cite{AJRN2014} 
and therefore omitted.
We just note here that the approximations \eqref{Th1eq} and \eqref{Th2eq} are
obtained by applying the Laplace method
(cf.\ \cite{Jensen1994}). Roughly speaking the Laplace method employs a 
second order expansion of the exponent of the integrand defining 
$\Exp[X^k\e^{-\theta X}]$
around a specific value.
The first approximation is obtained using the standard Laplace's method 
with an expansion around the value maximizing the exponent 
$\rho_k=\LW(\theta\sigma^2\e^{k\sigma^2})$ while the 
second one is obtained instead by using an expansion around the value
$\rho\ =\ -\LW(\theta\sigma^2)$.


Next we focus on finding asymptotic approximations for the derivatives
of the cumulant transform $\kappa(\theta)$. Recall that such derivatives
are associated to the Esscher transform and the
distributions in the exponential family generated
by the lognormal via \eqref{10.2b}. Combining this with
Proposition~\ref{Th1},
we arrive at the following asymptotic equivalences:
\begin{Corollary}\label{Cor5}
 Let $X\dLNormal(0,\sigma^2)$ and define
 \begin{equation}\label{musigma}
 \mu_\theta=-\LW(\theta\sigma^2),\qquad
 \sigma^2_\theta=\frac{\sigma^2} {1+\LW(\theta\sigma^2)}.
\end{equation}
Then
\begin{equation*}
 \lim_{\theta\rightarrow\infty}\dfrac{\Exp_\theta[X]}{\exp\big\{\mu_\theta+{\sigma^2_\theta}/{2}\big\}}=1,\qquad
 \lim_{\theta\rightarrow\infty}\dfrac{\Var_\theta[X]}
 {\exp\big\{2\mu_\theta+{\sigma^2_\theta}\big\}  \big(\e^{\sigma_\theta^2}-1\big)}=1.
\end{equation*}
where $\Exp_\theta$ and $\Var_\theta$ are the expectation and variance operators
under $F_\theta$.
\end{Corollary}
\begin{proof}
Using \eqref{Th2eq} in Proposition~\ref{Th1} we arrive at
\begin{equation*}
 \Exp_\theta[X]=  \exp\big\{\mu_\theta+{\sigma^2_\theta}/{2}\big\}
     \frac{\Exp\big[g_1((Z+1)\sigma_{1,\theta})\big]}
     {\Exp\big[g_0(Z\sigma_{0,\theta})\big]}
\end{equation*}
and
\begin{equation*}\label{Var.Exp}     
 \Var_\theta[X]=\exp\big\{2\mu_\theta+{\sigma^2_\theta}\big\}
     \bigg[\e^{\sigma^2_\theta}\frac{\Exp\big[g_2((Z+2)\sigma_{2,\theta})\big]}
     {\Exp\big[g_0(Z\sigma_{0,\theta})\big]}
     -\frac{\Exp^2\big[g_1((Z+1)\sigma_{1,\theta})\big]}
	 {\Exp^2\big[g_0(Z\sigma_{0,\theta})\big]}\bigg]
\end{equation*}
where $Z\dNormal(0,1)$ and the function $g_0, g_1, g_2$ are defined as 
in Proposition \ref{Th1}. The result follows as a consequence
of \eqref{10.2c}.
\end{proof}

Interestingly, one could at least at the heuristic level
identify $F_\theta$ as approximate lognormal with parameters
\begin{equation*}
 \mu_\theta=-\LW(\theta\sigma^2),\qquad
 \sigma^2_\theta=\frac{\sigma^2} {1+\LW(\theta\sigma^2)}.
\end{equation*} 
Thus we conclude that a sensible approximation of the Esscher transform
of a lognormal distribution is again a lognormal distribution whose parameters 
are given explicitly in terms 
of the Lambert W function.  Moreover, the expectation and variance of such a lognormal
random variable coincide with the values given in formula \eqref{musigma}.

The argument is to use
the Laplace method to get
\begin{equation*}
 F_\theta(x)=\frac{\Exp[\e^{-\theta X}\,\Ind(X<x)]}{\Exp[\e^{-\theta X}]}
 =\frac{\Exp\big[g_0(W)\,\Ind(W<\log x-\mu_\theta)\big)\big]}{\Exp[g_0(W)]},\qquad x>0,
\end{equation*}
where $W\dNormal(0,\sigma^2_\theta)$ and $g_0$ is defined as 
in Proposition~\ref{Th1}.
The Laplace method constructs the approximation above
in such way that the function 
$g_0$ is close enough to $1$ around a neighborhood of $0$ where
the mass of the random variable $W$ is concentrated.  
Hence, neglecting the error associated to the function $g_0$ 
and normalizing we arrive at
\begin{equation*}
 F_\theta(x)\approx\Prob\big(W<\log x-\mu_\theta\big).
\end{equation*} 

In Fig.~\ref{Fig0},  the right solid line blue plot is the LN$(0,\sigma^2)$ density
with $\sigma=0.25$. The other solid line blue plots  are (from right to left)
the $F_\theta$-densities for $\theta=10,25,100$. It is notable how little
even such a large values as $\theta=100$ shifts the distribution towards the origin,
which can be explained by the 
lognormal density decaying only slowly to 0 as $x\downarrow 0$.
The two dotted red plots are the lognormal approximations of the $F_\theta$-densities
for $\theta=25$ and 100. In Section~\ref{SS:Sim_f_theta} we derive an
alternative approximation in terms of the Gamma distribution.

\begin{figure}[ht]
   \caption{$F_\theta$-densities and their lognormal approximations}
   \label{Fig0}
   \centering
     \includegraphics[width=0.6\textwidth]{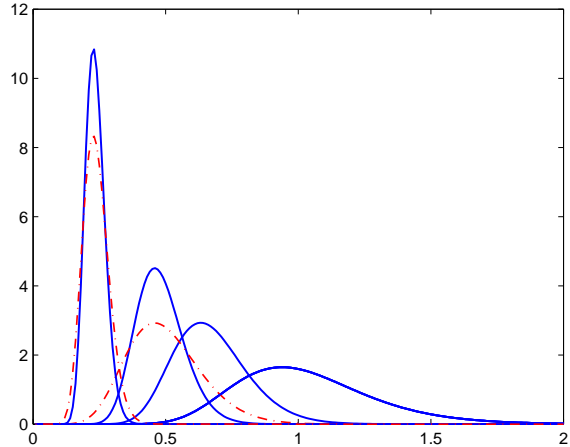}
 \end{figure}

%
%

\subsection{The Cram\'er function}
The previous results will allow us to provide asymptotic approximations 
for the Cram\'er function of the lognormal
distribution $\theta(x)$, i.e.\ the solution of the
equation $\Exp_{\theta(x)}[X]=x$.
Our suggestion is
\begin{equation}\label{Cramer}
 \widetilde{\theta}(x)=\frac{\gamma(x)\e^{\gamma(x)}}{\sigma^2}\qquad
 \text{where}\qquad
  \gamma(x)
   =\frac{-1-\log x+\sqrt{(1-\log x)^2+2\sigma^2}}{2}.
 \end{equation}
 
 To arrive at \eqref{Cramer}, we 
 approximate the value $\theta(x)$ solving the equation $\Exp_\theta[X]=x$  
with the solution $\widetilde{\theta}(x)$ of  
\begin{equation}\label{25.2a}
\exp\{\mu_\theta+\sigma^2_\theta/2\}=x
\end{equation} 
with $\mu_\theta$ and $\sigma^2_\theta$ as given in \eqref{musigma}.
Writing $\LW=\LW(\theta\sigma^2)$ for the moment,
\eqref{25.2a} then means
\[-\LW+\frac{\sigma^2/2}{1+\LW}\ =\ \log x\]
or equivalently that $\gamma=\LW$ is the solution of the quadratic
$\gamma^2+(1+\log x)\gamma-\sigma^2/2+\log x=0$. This gives
\eqref{Cramer} (excluding the negative sign of the square root by an easy argument).
%
%

By easy calculus,
\begin{equation}\label{27.2a}
 \gamma(x)=-\log x+\oh(1) \text{\ \ as\ }x\rightarrow0.
 \end{equation}
In consequence, the approximated solution $\widetilde{\theta}(x)$ is asymptotically equivalent to
 \begin{equation*}
   \widetilde{\theta}(x)\sim-\frac{\log x}{x\sigma^2},\qquad x\rightarrow0.
 \end{equation*} 
In Sections~\ref{S:Saddle} and~\ref{SS:IS} we will employ these results
to construct a saddlepoint approximation and a Monte Carlo estimator
of the left tail probability of a sum of lognormal random variables.
In particular, the asymptotic results derived above will be useful to show that
the approximation $\widetilde{\theta}(x)$ is asymptotically sharp and such that
when used as the twisting parameter of an exponential change of measure
estimator it remains asymptotically efficient as $x\rightarrow0$.


\section{Saddlepoint approximation in the left tail 
of a lognormal sum}\label{S:Saddle}


Daniels' saddlepoint method produces an approximation of the 
density function of a sum of i.i.d.\ random variables which is 
valid asymptotically on the number
of summands. The first and second order approximations are embodied in 
the formula
\begin{equation*}
 f_n(nx)\approx\sqrt{\frac{1}{2\pi n\kappa^{\prime\prime}(\theta(x))}} 
  \exp\big\{n\kappa^\ast(x)\big\}
 \big(1+\frac{1}{n}[\zeta_4(\theta(x))/8+5\zeta_3(\theta(x))^2/24]  \big),
\end{equation*}
where $\kappa^\ast(x)=\kappa(\theta(x))+x\theta(x)$ is the convex 
conjugate of $\kappa(x)$ and 
\begin{equation*}
 \zeta_k(\theta)=\frac{\kappa^{(k)}(\theta)}{\kappa^{\prime\prime}(\theta)^{k/2}},
\end{equation*} 
is the standardized cumulant.

The corresponding saddlepoint approximation for the 
cumulative distribution function is given by 
(Jensen, 1995, \cite{Jensen1994})
\begin{align*}
 F_n(nx)& =
  \frac{1}{\lambda_n(x)} 
   \exp\big\{n\kappa^\ast(x)\big\}
 \Bigl\{
   B_0(\lambda_n(x))
 \\ & \ \ 
  + \frac{\zeta_3(\theta(x))}{6\sqrt n}B_3(\lambda_n(x))
  + \frac{\zeta_4(\theta(x))}{24 n}B_4(\lambda_n(x))
  + \frac{\zeta_3(\theta(x))^2}{72 n}B_6(\lambda_n(x))
 \Bigr\},
\end{align*}
where
\begin{equation*}
 \lambda_n(x)=\theta(x)\sqrt{n\kappa^{\prime\prime}(\theta(x))},\quad 
 B_0(\lambda) =\lambda\e^{\lambda^2/2}\Phi(-\lambda),
\end{equation*}
and 
\begin{align*}
  B_3(\lambda) & =
  -\bigl\{\lambda^3 B_0(\lambda)-(\lambda^3-\lambda)/\sqrt{2\pi} \bigr\},\\
 \qquad
  B_4(\lambda) & =
  \lambda^4 B_0(\lambda)-(\lambda^4-\lambda^2)/\sqrt{2\pi},
 \\
  B_6(\lambda) & =
  \lambda^6 B_0(\lambda)-(\lambda^6-\lambda^4+3\lambda^2)/\sqrt{2\pi}.
\end{align*}
General results for the saddlepoint approximation state that 
for a fixed $x$ the relative 
error is $O(1/n)$ for the first order approximation and $O(1/n^2)$ 
for the second order approximation. More can be said, however, for the 
case of a lognormal sum. It is simple to see that the density $f(x)$ 
is logconcave for $x<e^{1-\sigma^2}$ (second derivative of $\log(f(x))$ is 
negative) and according to Jensen (1995, section 6.2) \cite{Jensen1994} 
we therefore have that the saddlepoint 
approximations have the stated relative errors uniformly for $x$ 
in a region around zero. Furthermore, write the exponentially tilted 
density as $\exp(-h(x)-\kappa(\theta))$ with 
$h(x)=\log(x)+(\log(x))^2/(2\sigma^2)+\theta x$. We center and scale this 
density using $w=\LW(\theta\sigma^2)$ as follows
\begin{align*}
 h_0(u)& =h(e^{-w}(1+\frac{\sigma}{\sqrt w}u))-h(e^{-w})
 \\ & =
 \log(1+\frac{\sigma}{\sqrt w}u)+\frac{1}{2\sigma^2}\{
 [-w+\log(1+\frac{\sigma}{\sqrt w}u)]^2-w^2\}+
 \frac{\sqrt w}{\sigma}u.
\end{align*}
From this we find $h_0(u)=\frac{1}{2}u^2+O(\sigma (|u|+|u|^3)/\sqrt w)$. 
Thus, for $w\rightarrow\infty$ ($\theta\rightarrow\infty$), the 
density converges uniformly in the region $|u|<(\sqrt w/\sigma)^{1/6}$, say, 
to the standard normal density. Due to the logconcavity the left tail 
beyond $-(\sqrt w/\sigma)^{1/6}$ is well behaved and for the right tail 
we find $h_0^{\prime}(u)>(\sqrt w/\sigma)^{1/6}/2$ for $\sigma/\sqrt w<1$. 
The convergence to the standard normal density as $w\rightarrow\infty$ 
implies that the saddlepoint approximations become exact in the 
limit $w\rightarrow\infty$. 

To evaluate the saddlepoint approximation we need to calculate the 
Laplace transform and its derivatives numerically. We want to implement 
the integration in such a way that the relative accuracy 
of the integration is of the 
same order irrespective of the argument $\theta$. 
For $k=0,1,2,3,4$ we want to evaluate the integral 
\[
 L_k(\theta)=E(X^k\exp(-\theta X))=\int_{-\infty}^\infty
 \frac{1}{\sqrt{2\pi\sigma^2}}\exp(-h(y))dy,
 \quad
 h(y)=\theta e^y+\frac{1}{2\sigma^2}y^2-ky
\]
The minimum of $h(y)$ is attained at $y_0=k\sigma^2-w$, where 
$w=\LW(\theta\sigma^2\exp(k\sigma^2))>0$ and 
$h(y_0)=w^2/(2\sigma^2)+w/\sigma^2-\frac{1}{2}\sigma^2k^2$. 
Furthermore, 
\[
 h_0(z)=h(y_0+z)-h(y_0)=\frac{w}{\sigma^2}(e^z-1-z)+\frac{z^2}{2\sigma^2}.
\]
Since $h_0^{\prime\prime}(z)=(we^z+1)/\sigma^2>0$ we see that 
$h$ is convex. Choosing a scale $\tau$ such that $2h_0(-\tau)= 1$ 
we obtain that $h_1(u)=2h_0(\tau u)$ is a convex function bounded 
between 0 and 1 for $-1<u<0$, is above $-u$ for $u<-1$ and with 
$h_1(u)\geq h_1(-u)$ for $u>0$. In this way the precision of the 
numerical integration of $\exp(-\frac{1}{2}h_1(u))$ will be of the same order 
irrespective of the value of $w$ and $\sigma^2$. 
In practice we can take $\tau$ as 
\[
 \tau=\begin{cases}
 \dfrac{\sigma}{\sqrt{1+w}}, & \dfrac{\sigma}{\sqrt{1+w}}\leq c_0,
 \\[.75cm]
 \sqrt{w^2+2w+\sigma^2}-w, & \dfrac{\sigma}{\sqrt{1+w}}> c_0,
 \end{cases}
\]
where $c_0$ is an arbitrary constant. Unless $\sigma^2$ is large we can 
use $\tau=\sigma/\sqrt{1+w}$ for all $w$. 

The saddlepoint $\theta(x)$ being the solution to 
$\kappa^\prime(\theta(x))=-x$ can be found by Newton-Raphson iteration 
using $\tilde\theta(x)$ as the initial value. 
In Table \ref{TableTildeTh} the initial value $\tilde\theta(x)$ 
is given together with $\theta(x)$ and the mean value 
$E_{\tilde\theta(x)}(X)$. In all cases in the table four steps in 
the Newton-Raphson suffices for reaching $\theta(x)$. 

In Tables \ref{MyTableCDF1} and \ref{MyTableCDF2} are 
examples with the saddlepoint approximation 
to the left tail probability.

 \begin{table}[ht]
    \centering
    \footnotesize
      \caption{Comparison of initial value $\tilde\theta(x)$ and 
final value $\theta(x)$ for the case $\sigma=0.250$.}
      \label{TableTildeTh} 
    \begin{tabular}{crrr}
      \toprule
       {$x$} & {$\Exp_{\widetilde{\theta}(x)}[X]$} & 
 $\tilde\theta(x)$ & $\theta(x)$ \\ 
      \midrule
 1.0 & 0.99905160 &  0.5002255 &  0.4850103 \\
 0.9 & 0.89695877 &  2.4295388 &  2.3625893 \\
 0.8 & 0.79589537 &  5.0894397 &  4.9624633 \\
 0.7 & 0.69554784 &  8.8690980 &  8.6691868 \\
 0.5 & 0.49617443 & 23.1845282 & 22.7639315 \\
 0.3 & 0.29767635 & 65.8850274 & 64.9626105 \\
 0.1 & 0.09934273 &373.4301331 & 369.9235664 \\
	   \bottomrule
    \end{tabular}
  \end{table}

  \begin{table}[ht]
    \centering
    \footnotesize
      \caption{Approximation of the CDF of a lognormal sum 
with $n=4$ and $\sigma=0.25$.} 
      \label{MyTableCDF1}
    \begin{tabular}{cccccc}
      \toprule
       {$x$} & {${\theta}(x)$} & {Saddle0} & {Saddle1} & {MC/AA}& {MC/MC}\\ 
      \midrule
        0.6500 &  11.132319 & 0.0001536084 & 0.0001592339 &
 1.61e-04 $\pm$ 1.76e-06 & 1.61e-04 $\pm$ 1.83e-06 \\ 
        0.7000 & 8.669187 & 0.0012499087 & 0.0013015022 &
 1.31e-03 $\pm$ 1.34e-05 & 1.30e-03 $\pm$ 1.37e-05 \\ 
        0.7500 &  6.644334 & 0.0065782847 & 0.0068830734 &
 6.93e-03 $\pm$ 6.50e-05 & 6.92e-03 $\pm$ 6.66e-05 \\ 
        0.8000 & 4.962463 & 0.0242679549 & 0.0255206432 &
 2.56e-02 $\pm$ 2.20e-04 & 2.55e-02 $\pm$ 2.24e-04 \\ 
        0.8500 & 3.552969 & 0.0669477011 & 0.0707464921 &
 7.13e-02 $\pm$ 5.53e-04 & 7.11e-02 $\pm$ 5.61e-04 \\ 
        0.9000 & 2.362589 & 0.1456850237 & 0.1545557418 &
 1.55e-01 $\pm$ 1.09e-03 & 1.55e-01 $\pm$ 1.10e-03 \\ 
	   \bottomrule
    \end{tabular}
  \end{table}
  \begin{table}[ht]
    \centering
    \footnotesize
      \caption{Approximation of the CDF of a lognormal sum with $n=64$ and $\sigma=0.25$.} 
      \label{MyTableCDF2}
    \begin{tabular}{cccc}
      \toprule
       {$x$}  & {${\theta}(x)$} & {Saddle0} & {Saddle1}\\ 
      \midrule
 0.90 & 2.3625893 & 8.693420e-06 & 8.772302e-06 \\
 0.91 & 2.1470381 & 3.951385e-05 & 3.989503e-05 \\
 0.92 & 1.9383125 & 1.575592e-04 & 1.591772e-04 \\
 0.93 & 1.7361482 & 5.538798e-04 & 5.599406e-04 \\
 0.95 & 1.3505093 & 4.782814e-03 & 4.842303e-03 \\
 0.97 & 0.9882486 & 2.646345e-02 & 2.683567e-02 \\
 0.99 & 0.6476640 & 9.774927e-02 & 9.926919e-02 \\
	   \bottomrule
    \end{tabular}
  \end{table}


\section{Simulation}\label{S:Sim}


\subsection{Random variate generation from $F_\theta$}\label{SS:Sim_f_theta}
We first consider the problem of generating a r.v.\ from the density
\begin{equation*}
 f_\theta(x)={\e^{-\theta x-\kappa(\theta)}f(x)},\qquad x>0.
\end{equation*}

The obvious naive choice is acceptance-rejection
(A-R; \cite[II.2]{AsmussenGlynn2007}), simulating $Z$ from $f$
and rejecting w.p.\ $\e^{-\theta Z}$. This choice produces a very simple algorithm
for generating from $f_\theta$ and the method is exact even 
when we do not have an explicit 
expression for $\kappa(\theta)$. 
\begin{Algorithm}\label{Alg.Exp.Log.Log}
Simulate $U\sim\mathrm{U}(0,1)$, $Z\sim\mathrm{LN}(0,\sigma^2)$.
If $U>\e^{-\theta Z}$ repeat.  Else, return $Z$.
\end{Algorithm}
Ideally,
we would like to have a rejection probability $p$ as close to 1 but in our case 
$p=\e^{\kappa(\theta)}$, so as the value of $\theta$ increases, the 
probability of acceptance diminishes
and hence the expected number $c$ of rejection steps goes to $\infty$.  In consequence, this estimator
is very inefficient for large values of $\theta$.

An alternative algorithm is as follows.  Observe that 
$f_\theta(x)$ is proportional to $\exp\{-\theta x-\log^2x)/2\sigma^2-\log x\}$ which
is maximized at $x=m=w/\theta\sigma^2$ where 
$w=\LW(\theta\sigma^2\e^{-\sigma^2})$. Let $Y_\theta\sim f_.$
Up to a constant, the density of $Y_\theta/m$ is
\begin{align}
 f(mx)&=\exp\{-\theta mx-\log^2(mx)/2\sigma^2-\log(mx)\}\nonumber\\
      &\propto\exp\{-\sigma^{-2}[wx+\log^2m/2+\log^2(x)/2+\log m\log x+\sigma^2\log x]\}\nonumber\\
      &\propto\exp\{-\sigma^{-2}[wx+\log^2(x)/2-w\log x]\}\nonumber\\
       &=x^{\sigma^{-2}w}\e^{-\sigma^{-2}wx}\e^{-\log^2(x)/2\sigma^2}\label{Gamma}.
\end{align}
Note that $-w=\log m+\sigma^2$. This gives the following A-R algorithm:
\begin{Algorithm}\label{Alg.Exp.Log}
Simulate $U\sim\mathrm{U}(0,1)$ and $Z\sim\mathrm{Gamma}(w/\sigma^2+1,w/\sigma^2)$.
If $U>\e^{-\log^2(Z)/2\sigma^2}$ repeat.  Else, return $Y_\theta=mZ=wZ/\theta\sigma^2$.
\end{Algorithm}
Now the mean of $Z$ is $1+\sigma^2/w$ and the variance is 
$\omega_\theta^2=(1+\sigma^2/w)\sigma^2/w$.  Since $\omega_\theta^2\rightarrow0$
as $\theta\rightarrow\infty$ and hence $w\rightarrow\infty$, $Z\rightarrow1$ in probability.
Hence, by dominated convergence, the overall probability 
$\Exp[\e^{-\log^2(Z)/2\sigma^2}]$
converges to 1!

\begin{Remark}({\sc Gamma approximation})
\emph{A-R} with acceptance probability $\rightarrow1$ is only possible if the ratio of the
proposal and target density goes to 1.  Thus, $Y_\theta$ is accurately approximated by 
$wZ/\theta\sigma^2\sim\mathrm{Gamma}(w/\sigma^2+1,\theta)$.  And if we pass to sums,
$S_n=Y_{\theta,1}+\dots+Y_{\theta,n}$ is accurately approximated by the Gamma r.v.\ 
$w(Z_1+\dots+Z_n)/\theta\sigma^2$ which is 
$\mathrm{Gamma}(n(w/\sigma^2+1),\theta)$.
\end{Remark}

\begin{Remark} ({\sc Normal approximation})
 Note that to first order, $w\sim\log\theta$.  Thus the variance $(w+\sigma^2)/\theta^2\sigma^2$
 goes to $0$, and we can approximate $Y$ by a r.v.\ distributed as
 $\mathrm{N}\big((w+\sigma^2)/\theta\sigma^2),(w+\sigma^2)/\theta^2\sigma^2\big)$ 
 and $S_n$ with 
 $\mathrm{N}\big(n(w+\sigma^2)/\theta\sigma^2),n(w+\sigma^2)/\theta^2\sigma^2\big)$.
\end{Remark} 

When plotted on top of Fig.~\ref{Fig0}, the two Gamma approximations
for $\theta=25$ and 100 are indistinguishable from the lognormal$(\mu_\theta,
\sigma^2_\theta)$ approximations.

\subsection{Efficient Monte Carlo for left tails of lognormal sums}
\label{SS:IS}

In this section we develop an \emph{asymptotically efficient Monte Carlo estimator} $\widehat\alpha_n(x)$,
for the left tail probability of a lognormal sum $\alpha_n(x)=\Prob(S_n\le nx)$ 
 as $x\rightarrow0$.
 
 We start by recalling some standard concepts from rare event simulation
 (\cite[VI.1]{AsmussenGlynn2007}).
In our setting, we say that a Monte Carlo estimator $\widehat\alpha_n(x)$ is 
\emph{strongly efficient} or has \emph{bounded relative error} as $u\rightarrow0$ if 
\begin{equation*}
 \limsup\limits_{x\rightarrow0}\frac{\Var\widehat \alpha_n(x)}{\alpha_n^2(x)}<\infty.
\end{equation*}
This efficiency property implies that the number of replications required to 
estimate $\alpha_n(nx)$ with certain fixed relative precision remains bounded as $x\rightarrow0$.
A weaker criterion is \emph{logarithmic efficiency} 
defined as
\begin{equation*}
 \limsup\limits_{x\rightarrow0}\frac{\Var
  \widehat \alpha_n(x)}{\alpha_n^{2-\epsilon}(x)}=0,\qquad\forall\epsilon>0.
\end{equation*}
From a practical point of view, there is no substantial difference between
these two criteria.  However, it is often easier to prove
logarithmic efficiency rather than bounded relative error. 
Logarithmic efficiency implies that the number of replications needed for achieving 
certain relative precision grows at  rate of order at most $|\log(\alpha_n(x))|$.

An alternative unbiased estimator can be obtained by using the variance reduction
technique \emph{importance sampling}
(\cite[V.1]{AsmussenGlynn2007}).  This
method relies on the existence of a Radon-Nikodym derivative with respect
to a probability measure, say $\mathbb{Q}$.  If we are interested
in estimating $\Exp[h(W)]$ where $W$ is random, $h$ an measurable function, $\Exp$
is the expectation operator under the measure $\Prob$ and
$\mathbb{Q}$ is an absolutely continuous measure with respect to $\Prob$,
then it holds that
\begin{equation*}
 \Exp[h(W)]=\Exp^{\mathbb{Q}}[L\,h(W)],
\end{equation*}
where $\Exp^{\mathbb{Q}}$ is the expectation operator under the measure $\mathbb{Q}$
and $L=\dd \Prob/\dd \mathbb{Q}$ is the Radon-Nikodym derivative of $\Prob$
with respect to $\mathbb{Q}$ (the last also goes under the name \emph{likelihood ratio}
in the simulation community).
Hence, if $X$ is simulated according to $\mathbb{Q}$, then
$L\,h(W)$ serves as an unbiased estimator of the quantity
$\Exp[h(W)]$.
The strategy of selecting an importance distribution from 
the exponential family generated by the lognormal $\{F_\theta:\theta\in\Theta\}$
is often referred as \emph{exponential twisting},
\emph{exponential tilting} or simply \emph{exponential change of measure}.  
Ideally, 
the twisting parameter $\theta$ is selected as the 
value of the Cram\'er function $\theta(\cdot)$
evaluated at $x$, and defined via
\begin{equation}\label{optimalIS}
 -\kappa^\prime(\theta(x))=\Exp_{\theta(u)}[X]=x.
\end{equation}
Doing so, exponential twisting is logarithmically efficient  as $n\to\infty$,
cf.\ \cite[][p.\ 169-171]{AsmussenGlynn2007}.  

Notice, however, that difficulties arise in the right tail
if the $X_i$'s are heavy-tailed: then the integral 
associated with $\Exp[\e^{-\theta X}]$
diverges for negative values of the argument $\theta$ and in consequence, 
the equation \eqref{optimalIS} has no solution if $x>\Exp[X_i]$.
Further difficulties in the heavy-tailed environment are exposed
in \cite{AsmussenBinswangerHojgaard00,BassambooJunejaZeevi}.
Nevertheless, exponential twisting can be implemented 
for the left tail probability of a lognormal sum; moreover, it turns
out that it is logarithmically efficient:
\begin{theorem}\label{Beta}
  Consider $X_1,\dots,X_n\sim F_{\theta(u)}$ and
  set $S_n=X_1+\dots+X_n$.  Define 
  \begin{equation*}
    \beta_n(x):={\mathcal{L}}^n(\theta(x))\,\e^{\theta(x) S_n}\,
    \Ind\{S_n<nx\},
  \end{equation*}
  where ${\mathcal{L}}(\cdot)$ is the Laplace transform of the lognormal distribution.
  Then ${\beta}_n(x)$ is a logarithmically efficient and unbiased estimator of $\alpha_n(x)$ as $n\to\infty$.
\end{theorem} 
\begin{proof}
The lognormal density is log-concave and so the result follows 
immediately from the proof of Theorem 2.10, Chapter VI in \cite{AsmussenGlynn2007}.
\end{proof}

Notice however, that the optimal exponential twisting algorithm described 
above is not implementable since the Laplace transform
$\mathcal{L}(\theta)$ and the Cram\'er function $\theta(\cdot)$ are unknown.  
In this paper, we propose an alternative estimator of
$\alpha_n(x)$ which is logarithmic efficient as $x\rightarrow0$
(not necessarily as $n\to\infty$!).  
For its construction we employ the approximation of the Cram\'er function
provided in the previous section and we assume that an unbiased estimator
of the Laplace transform is available.  
The algorithm is as follows:

\begin{Algorithm}\
\begin{enumerate}
 \item Use the approximation $\widetilde\theta(x)$ of the Cram\'er function given in \eqref{Cramer}.
 \item Obtain an unbiased estimate $\widehat{\mathcal{L}}(\widetilde\theta(x))$ of the Laplace transform
   \cite[cf.][]{AJRN2014}
 \item Simulate $X_1,\dots,X_n\sim F_{\widetilde\theta(x)}$ and set $S_n=X_1+\dots+X_n$.
 \item Return 
  \begin{equation}\label{alpha1}
   \widehat\alpha_n(x)=\e^{\widetilde\theta(x) S_n}\big[\widehat{\mathcal{L}}
   (\widetilde\theta(x))\big]^n\,\Ind(S_n<nx).
  \end{equation}
\end{enumerate}
\end{Algorithm}
However, even if $\widehat{\mathcal{L}}(\theta)$ was unbiased, then 
the estimator \eqref{alpha1} is biased for $\alpha_n(x)$; because of
Jensen's inequality it holds that $\Exp[\widehat{\mathcal{L}}^n(\theta)]>\mathcal{L}^n(\theta)$.
A solution is to take the product of $n$ independent copies of 
$\widehat{\mathcal{L}}(\theta)$ to estimate ${\mathcal{L}}^n(\theta)$ without bias.
We go more into this topic in Section~\ref{S:Numerical} and 
Section~\ref{S:A:Ln} of the Appendix.

We discuss next an interesting asymptotic 
efficiency property of \eqref{alpha1} when we employ the logarithmically efficient unbiased estimator
of $\mathcal{L}(\theta)$ suggested in \cite{AJRN2014}.  That is
\begin{equation}\label{ISestimator}
  \widehat{\mathcal{L}}_{IS}(\theta)=\widetilde{\mathcal{L}}(\theta)  
  \,\vartheta(Y,\theta),  
\end{equation}
where $Y\sim\mathrm{N}(0,\sigma^2)$,
\begin{equation*}
\widetilde{\mathcal{L}}(\theta)=
  \exp\biggl\{-\frac{\LW^2(\theta\sigma^2)+2\,\LW(\theta\sigma^2)}{2\sigma^2}\biggr\},
 \qquad
   \vartheta(t;\theta)=
  \exp\biggl\{-\frac{\LW(\theta\sigma^2)}{\sigma^2}
  \bigl(\e^{t}-1-t\bigr)\biggr\}.
\end{equation*}

\begin{Proposition}\label{Th5}
 Let 
 \begin{equation}\label{TheEstimator}
  \widehat\alpha_n(x)=\e^{\widetilde\theta(x) S_n}\big[\widehat{\mathcal{L}}_{\mathrm{IS}}
   (\widetilde\theta(x)\big]^n\,\Ind(S_n<nx)
 \end{equation}
 Then $\widehat\alpha_n(x)$ is a consistent estimator of
 $\alpha_n(x)$.  Moreover, it is also a logarithmic efficient estimator in 
the mean square error sense. That is, for all $\epsilon>0$
\begin{equation}\label{Eff}
 \limsup_{x\rightarrow0}\frac{\mathrm{MSE}_{\alpha_n(x)}[\widehat\alpha_n(x)]}{\alpha^{2-\epsilon}_n(x)}=0,
\end{equation}
\end{Proposition}
For the proof, see Section~\ref{S:A:ProofTh5} of the Appendix.

\subsection{Density estimation}

Consider the problem of estimating the density of a lognormal sum via 
simulation. Following \cite{AsmussenGlynn2007}, Example V.4.3 p.\,146, slightly
extended,
we first note that the conditional density at $nx$ of $S_n$ given 
\[S_{n,-i}\ =\ X_1+\cdots+X_{i-1}+X_{i+1}+\cdots+X_n\ =\ S_n-X_i\]
is $f(nx-S_{n,-i})$. Hence an unbiased estimator of $f_n(nx)$ is
$\sum_{1}^nf(nx-S_{n,-i})/n$.

However, since we are dealing with values of $x$ far to the left of
$\Exp X$, it is likely that $S_{n,-i}>x$ so that $f(nx-S_{n,-i})=0$ 
and the procedure will come out with a large number of
zeroes. Hence we employ the same importance sampling estimator
as used elsewhere. That is, we simulate the $X_j$ from $f_{\tilde\theta(x)}$
and return the estimator
\[\widehat{f_n}(nx)\ =\ \frac{\exp\{\tilde\theta(x)S_n+n\kappa\bigl(\tilde\theta(x)\bigr)\}}{n}\sum_{i=1}^nf(nx-S_{n,-i})\]
(in practice to be averaged over $R$ replications). 
An alternative slightly more complicated estimator is
\[\widehat{f_n}(nx)\ =\ \frac{1}{n}\sum_{i=1}^nf(nx-S_{n,-i})
\exp\{\tilde\theta(x)S_{n,-i}+(n-1)\kappa\bigl(\tilde\theta(x)\bigr)\}\,.\]

In \cite{GT2013}, an importance sampling estimator for $F_n(z)$ is suggested
and it is written that a parallel estimator for $f_n(z)$ can be constructed in the
same way. We do not follow the details of this statement.

\section{Numerical examples}
\label{S:Numerical}

In our numerical experiments, we have taken parameter values that we
consider realistic from the point of view of financial applications.
A yearly volatility of order $0.25$ is often argued to be typical.
We have considered periods of lengths one year, one quarter, one month and one week,
corresponding to $\sigma=0.25$, $\sigma=0.25/\sqrt{4}=0.125$,
 $\sigma=0.25/\sqrt{12}=0.072$,
resp.\ $\sigma=0.25/\sqrt{52}=0.035$.
Real-life portfolios are often large, even in the thousands; 
the values we have chosen are $n=4,16,64,256$.

For each combination of $n$ and $\sigma$ we have conducted 
several numerical empirical analyses.  In all numerical experiments involving 
simulation we have employed $R=100,000$ replications.

\subsubsection*{Transformations associated to the Laplace transform}
We test empirically the approximations for 
the $n$-th power of the Laplace transform $\mathcal{L}^n(\theta)$.
The approximations discussed here are
alternatives to numerical integration.  
We considered two 
approximations: 
the $n$-th power of the approximation
derived from \eqref{10.2a}; and the $n$-th power of the IS sampling estimator \eqref{ISestimator}.  
Notice however that the last estimator is biased for any $n>1$ so
the bias will grow exponentially with $n$.  To address this issue
we considered two alternatives (see Appendix \ref{S:A:Ln} for further details): 
insert a bias correction term and to consider the unbiased estimator 
built as the product of $n$ independent copies 
of \eqref{ISestimator}. 
  
   \begin{table}[ht]
    \centering
    \footnotesize
      \caption{Approximated values of the $n$-th power of the Laplace transform with $n=256$ and $\sigma=0.250$.} 
      \label{tableLaplaceN1}
   \begin{tabular}{ccccc}
      \toprule
     {$\theta$} & {$\widetilde{\mathcal{L}}^n(\theta)$} & {$\widehat{\mathcal{L}}^n(\theta)$} & {$\widehat{\mathcal{L}}^n(\theta)- \text{Correction}$}& {$\widehat{\mathcal{L}^n}(\theta)$}  \\ 
      \midrule
           0.9705 & 1.23e-108 & 1.10e-108$\pm$3.14e-111 & 1.10e-108$\pm$3.14e-111 & 1.12e-108$\pm$4.92e-111 \\ 
           0.9010 & 4.10e-101 & 3.82e-101$\pm$9.86e-104 & 3.82e-101$\pm$9.86e-104 & 3.75e-101$\pm$1.53e-103 \\ 
           0.8322 & 1.20e-93 & 1.08e-93$\pm$2.71e-96 & 1.08e-93$\pm$2.71e-96 & 1.10e-93$\pm$4.11e-96 \\ 
           0.7642 & 3.08e-86 & 2.84e-86$\pm$6.49e-89 & 2.84e-86$\pm$6.49e-89 & 2.87e-86$\pm$9.84e-89 \\ 
           0.6971 & 6.95e-79 & 6.85e-79$\pm$1.35e-81 & 6.85e-79$\pm$1.35e-81 & 6.49e-79$\pm$2.02e-81 \\ 
           0.6307 & 1.38e-71 & 1.27e-71$\pm$2.45e-74 & 1.27e-71$\pm$2.45e-74 & 1.29e-71$\pm$3.63e-74 \\ 
           0.5651 & 2.40e-64 & 2.31e-64$\pm$3.89e-67 & 2.31e-64$\pm$3.89e-67 & 2.26e-64$\pm$5.71e-67 \\ 
           0.5002 & 3.69e-57 & 3.45e-57$\pm$5.36e-60 & 3.45e-57$\pm$5.36e-60 & 3.49e-57$\pm$7.83e-60 \\ 
	   \bottomrule
   \end{tabular}
  \end{table} 
We consider an example with large $n=256$ and $\sigma=0.25$; the number of replications
for all estimators was $R=100,000$.  The results can be found
in Table \ref{tableLaplaceN1}.  
It is notorious that for moderate values of the 
parameter $\theta$ one obtains very small values of the Laplace transform.
Also, the numerical results of this example indicate that the approximation 
$\widetilde{L}(\theta)$ underestimates
the real value of $\mathcal{L}(\theta)$.  
It is also noted that the Bias Correction (BC) does not provide a significant improvement
over the value of the approximation.  In particular, the variance of the estimator
$\widehat{\mathcal{L}^n}(\cdot)$ is only slightly larger that the variance of 
$\widehat{\mathcal{L}}^n(\cdot)$ and its bias appears to be very small.
In spite of these minor pitfalls, we consider all the estimators to be very sharp.  
However, $\widehat{\mathcal{L}^n}(\cdot)$ is unbiased so the 
undesired amplifying effect on the bias produced by the $n$-th power transformation
is avoided. We favor the use of the estimators $\widehat{\mathcal{L}^n}(\cdot)$
because of its unbiasedness and relatively small variance.

\subsubsection*{Left tail of the Lognormal Sum}

Next we verify the approximations for the cdf and pdf 
of the lognormal sum.  We have thereby been thinking of a portfolio of 
$n$  assets with next-period values $Y_1,\ldots,Y_n$ assumed i.i.d.\ 
lognormal$(\mu,\sigma^2)$, such that a loss corresponds
to a small value $x$ of $S_n=Y_1,\ldots,Y_n$. When choosing $x$,
we have had the recommended VaR values 0.99\%--0.99.97\%
of Basel II \cite{BaselII} in mind and chosen
$\Prob(S_n\le nx)$ to be in the interval 0.0001--0.0100.

We have proposed two type of approximations:  saddlepoint
approximations and Monte Carlo estimators.  We start with 
the saddlepoint approximation for both the pdf and cdf of the sum of $n$ lognormals
and which are given by
\begin{align*}
 \widetilde{F_n}(nx)&= 
   \exp\big\{n\kappa^\ast(x)\big\}\,{\e^{\lambda_n(x)^2/2}\,\Phi(-\lambda_n(x))},&
  \widetilde{f_n}(nx)&=\frac{\exp\big\{n\kappa^\ast(x)\big\}}
 {\sqrt{2\pi n\kappa^{\prime\prime}(\theta(x))}},&
\end{align*}
where $\lambda_n(x)=\theta(x)\sqrt{n\kappa^{\prime\prime}(\theta(x))}$ and 
$\kappa^\ast(\cdot)$ is the complex conjugate of $\kappa(\cdot)$.
Recall that $\exp\{n\kappa^\ast(x)\}=\mathcal{L}^n(\theta(x))\e^{x\theta(x)}$.
In our numerical results we report saddlepoint approximations 
(labeled Saddle/MC); and 2) Monte Carlo estimators (labeled
MC/MC).
The last is based on the proposed  importance sampling estimator where the importance sampling
distribution was selected to be from the exponential family.  The parameter $\theta$
defining such distribution was selected to be equal to the approximation of the 
Cram\'er function $\widetilde\theta(\cdot)$ evaluated at $x/n$. The general
estimator for the CDF of the lognormal sum has the form
  \begin{equation*}
    \widehat{F}_n(nx)={\mathcal{L}}^n(\widetilde\theta(x))\,\e^{\widetilde\theta(x) S_n}\,\Ind\{S_n<nx\},
  \end{equation*}
where $S_n=X_1+\dots+X_n$ and $X_1,\dots,X_n$ is a sample from the exponential
family.  Similarly, the MC estimator of the pdf of the lognormal sum has the form
\begin{equation*}
 \widehat{f}_n(nx)=\mathcal{L}^{n-1}(\widetilde\theta(x))\,\bigg[\frac1n\sum_{i=1}^n
 \e^{\widetilde\theta(x)S_{n,-i}}\,f(nx-S_{n,-i}).
 \bigg].
\end{equation*}

In our numerical results we have used $R=100,000$.  Tables \ref{TableCDF1}-\ref{TableCDF5}
contain the numerical results for the CDF of the lognormal sum for various combinations
of the parameters $\sigma$ and $n$.  
Results for the PDF are given in Tables \ref{TablePDF1}-\ref{TablePDF2}.
The approximations
become sharper as either $\sigma$ or $x$ tend to $0$.


  \begin{table}[ht]
    \centering
    \footnotesize
      \caption{Approximation of the CDF of a lognormal sum with $n=4$ and $\sigma=0.25$.} 
      \label{TableCDF1}
    \begin{tabular}{ccccccc}
      \toprule
       {$x$} & {$nx$} & {$\widetilde{\theta}(x)$}  & {Saddle/MC} & {MC/MC}\\ 
      \midrule
        0.6500 &     2.60 &    11.38 & 1.53e-04 & 1.61e-04 $\pm$ 1.83e-06 \\ 
        0.7000 &     2.80 &     8.87 & 1.24e-03 & 1.30e-03 $\pm$ 1.37e-05 \\ 
        0.7500 &     3.00 &     6.81 & 6.54e-03 & 6.92e-03 $\pm$ 6.66e-05 \\ 
        0.8000 &     3.20 &     5.09 & 2.40e-02 & 2.55e-02 $\pm$ 2.24e-04 \\ 
        0.8500 &     3.40 &     3.65 & 6.59e-02 & 7.11e-02 $\pm$ 5.61e-04 \\ 
        0.9000 &     3.60 &     2.43 & 1.44e-01 & 1.55e-01 $\pm$ 1.10e-03 \\ 
	   \bottomrule
    \end{tabular}
  \end{table}
  \begin{table}[ht]
    \centering
    \footnotesize
      \caption{Approximation of the CDF of a lognormal sum with $n=64$ and $\sigma=0.25$.} 
      \label{TableCDF2}
    \begin{tabular}{ccccccc}
      \toprule
       {$x$} & {$nx$} & {$\widetilde{\theta}(x)$} & {Saddle/MC} & {MC/MC}\\ 
      \midrule
        0.9219 &    59.00 &     1.95 & 2.00e-04 & 2.04e-04 $\pm$ 2.56e-06 \\ 
        0.9336 &    59.75 &     1.71 & 8.31e-04 & 8.56e-04 $\pm$ 1.01e-05 \\ 
        0.9453 &    60.50 &     1.48 & 2.97e-03 & 3.06e-03 $\pm$ 3.35e-05 \\ 
        0.9570 &    61.25 &     1.26 & 8.99e-03 & 9.22e-03 $\pm$ 9.43e-05 \\ 
        0.9688 &    62.00 &     1.04 & 2.40e-02 & 2.45e-02 $\pm$ 2.30e-04 \\ 
        0.9805 &    62.75 &     0.83 & 5.45e-02 & 5.59e-02 $\pm$ 4.79e-04 \\ 
	   \bottomrule
    \end{tabular}
  \end{table}
 \begin{table}[ht]
    \centering
    \footnotesize
      \caption{Approximation of the CDF of a lognormal sum with $n=256$ and $\sigma=0.25$.} 
      \label{TableCDF3}
    \begin{tabular}{ccccccc}
      \toprule
       {$x$} & {$nx$} & {$\widetilde{\theta}(x)$} & {Saddle/MC} & {MC/MC}\\ 
      \midrule
        0.9727 &   249.00 &     0.97 & 9.96e-05 & 1.06e-04 $\pm$ 1.38e-06 \\ 
        0.9805 &   251.00 &     0.83 & 6.47e-04 & 6.75e-04 $\pm$ 8.16e-06 \\ 
        0.9844 &   252.00 &     0.76 & 1.52e-03 & 1.57e-03 $\pm$ 1.82e-05 \\ 
        0.9883 &   253.00 &     0.70 & 3.52e-03 & 3.43e-03 $\pm$ 3.79e-05 \\ 
        0.9922 &   254.00 &     0.63 & 6.70e-03 & 7.02e-03 $\pm$ 7.41e-05 \\ 
        1.0000 &   256.00 &     0.50 & 2.39e-02 & 2.50e-02 $\pm$ 2.37e-04 \\ 
	   \bottomrule
    \end{tabular}
  \end{table}  
  \begin{table}[ht]
    \centering
    \footnotesize
      \caption{Approximation of the CDF of a lognormal sum with $n=64$ and $\sigma=0.125$.} 
      \label{TableCDF4}
    \begin{tabular}{ccccccc}
      \toprule
       {$x$} & {$nx$} & {$\widetilde{\theta}(x)$} & {Saddle/MC} & {MCIS/MC}\\ 
      \midrule
        0.9500 &    60.80 &     3.98 & 8.37e-05 & 8.51e-05 $\pm$ 1.08e-06 \\ 
        0.9563 &    61.20 &     3.52 & 4.15e-04 & 4.16e-04 $\pm$ 4.97e-06 \\ 
        0.9625 &    61.60 &     3.06 & 1.70e-03 & 1.69e-03 $\pm$ 1.90e-05 \\ 
        0.9688 &    62.00 &     2.61 & 5.87e-03 & 5.88e-03 $\pm$ 6.15e-05 \\ 
        0.9750 &    62.40 &     2.17 & 1.73e-02 & 1.78e-02 $\pm$ 1.71e-04 \\ 
        0.9812 &    62.80 &     1.74 & 4.40e-02 & 4.45e-02 $\pm$ 3.94e-04 \\ 
	   \bottomrule
    \end{tabular}
  \end{table}
  \begin{table}[ht]
    \centering
    \footnotesize
      \caption{Approximation of the CDF of a lognormal sum with $n=64$ and $\sigma=0.072$.} 
      \label{TableCDF5}
    \begin{tabular}{ccccccc}
      \toprule
       {$x$} & {$nx$} & {$\widetilde{\theta}(x)$} & {Saddle/MC} & {MCIS/MC}\\ 
      \midrule
        0.9703 &    62.10 &     6.51 & 1.44e-04 & 1.43e-04 $\pm$ NaN \\ 
        0.9734 &    62.30 &     5.85 & 5.35e-04 & 5.37e-04 $\pm$ NaN \\ 
        0.9766 &    62.50 &     5.20 & 1.76e-03 & 1.78e-03 $\pm$ 1.98e-05 \\ 
        0.9797 &    62.70 &     4.55 & 5.17e-03 & 5.26e-03 $\pm$ 5.51e-05 \\ 
        0.9828 &    62.90 &     3.91 & 1.36e-02 & 1.37e-02 $\pm$ 1.34e-04 \\ 
        0.9859 &    63.10 &     3.28 & 3.17e-02 & 3.21e-02 $\pm$ 2.92e-04 \\ 
	   \bottomrule
    \end{tabular}
  \end{table}

  \begin{table}[ht]
    \centering
    \footnotesize
      \caption{Approximation of the PDF of a lognormal sum with $n=4$ and $\sigma=0.25$.} 
      \label{TablePDF1}
    \begin{tabular}{ccccccc}
      \toprule
       {$x$} & {$nx$} & {$\widetilde{\theta}(x)$} & {Saddle/MC} & {MC/MC} \\ 
      \midrule
        0.6500 &     2.60 &    11.38 & 1.90e-03 & 1.88e-03 $\pm$ 8.66e-06 \\ 
        0.7000 &     2.80 &     8.87 & 1.23e-02 & 1.22e-02 $\pm$ 5.62e-05 \\ 
        0.7500 &     3.00 &     6.81 & 5.15e-02 & 5.08e-02 $\pm$ 2.36e-04 \\ 
        0.8000 &     3.20 &     5.09 & 1.49e-01 & 1.47e-01 $\pm$ 6.85e-04 \\ 
        0.8500 &     3.40 &     3.65 & 3.16e-01 & 3.16e-01 $\pm$ 1.47e-03 \\ 
        0.9000 &     3.60 &     2.43 & 5.25e-01 & 5.24e-01 $\pm$ 2.46e-03 \\ 
	   \bottomrule
    \end{tabular}
  \end{table}
  \begin{table}[ht]
    \centering
    \footnotesize
      \caption{Approximation of the PDF of a lognormal sum with $n=64$ and $\sigma=0.25$.} 
      \label{TablePDF2}
    \begin{tabular}{ccccccc}
      \toprule
       {$x$} & {$nx$} & {$\widetilde{\theta}(x)$} & {Saddle/MC} & {MC/MC} \\ 
      \midrule
        0.9219 &    59.00 &     1.95 & 4.19e-04 & 4.12e-04 $\pm$ 5.73e-06 \\ 
        0.9336 &    59.75 &     1.71 & 1.55e-03 & 1.55e-03 $\pm$ 2.14e-05 \\ 
        0.9453 &    60.50 &     1.48 & 4.90e-03 & 4.83e-03 $\pm$ 6.72e-05 \\ 
        0.9570 &    61.25 &     1.26 & 1.29e-02 & 1.28e-02 $\pm$ 1.79e-04 \\ 
        0.9688 &    62.00 &     1.04 & 2.98e-02 & 2.91e-02 $\pm$ 4.05e-04 \\ 
        0.9805 &    62.75 &     0.83 & 5.72e-02 & 5.64e-02 $\pm$ 7.85e-04 \\ 
	   \bottomrule
    \end{tabular}
  \end{table}
  \begin{table}[ht]
    \centering
    \footnotesize
      \caption{Approximation of the PDF of a lognormal sum with $n=256$ and $\sigma=0.25$.} 
       \label{TablePDF3}
    \begin{tabular}{ccccccc}
      \toprule
       {$x$} & {$nx$} & {$\widetilde{\theta}(x)$} & {Saddle/MC} & {MCIS/MC} \\ 
      \midrule
        0.9727 &   249.00 &     0.97 & 1.03e-04 & 1.05e-04 $\pm$ 2.15e-06 \\ 
        0.9805 &   251.00 &     0.83 & 5.84e-04 & 5.86e-04 $\pm$ 1.21e-05 \\ 
        0.9844 &   252.00 &     0.76 & 1.27e-03 & 1.26e-03 $\pm$ 2.62e-05 \\ 
        0.9883 &   253.00 &     0.70 & 2.73e-03 & 2.57e-03 $\pm$ 5.30e-05 \\ 
        0.9922 &   254.00 &     0.63 & 4.78e-03 & 4.87e-03 $\pm$ 1.00e-04 \\ 
        1.0000 &   256.00 &     0.50 & 1.42e-02 & 1.45e-02 $\pm$ 2.98e-04 \\ 
	   \bottomrule
    \end{tabular}
  \end{table}
  \begin{table}[ht]
    \centering
    \footnotesize
      \caption{Approximation of the PDF of a lognormal sum with $n=64$ and $\sigma=0.125$.} 
            \label{TablePDF4}
    \begin{tabular}{ccccccc}
      \toprule
       {$x$} & {$nx$} & {$\widetilde{\theta}(x)$} & {Saddle/MC} & {MCIS/MC} \\ 
      \midrule
        0.9500 &    60.80 &     3.98 & 3.55e-04 & 3.58e-04 $\pm$ 4.82e-06 \\ 
        0.9563 &    61.20 &     3.52 & 1.58e-03 & 1.56e-03 $\pm$ 2.11e-05 \\ 
        0.9625 &    61.60 &     3.06 & 5.73e-03 & 5.70e-03 $\pm$ 7.70e-05 \\ 
        0.9688 &    62.00 &     2.61 & 1.73e-02 & 1.72e-02 $\pm$ 2.34e-04 \\ 
        0.9750 &    62.40 &     2.17 & 4.41e-02 & 4.49e-02 $\pm$ 6.05e-04 \\ 
        0.9812 &    62.80 &     1.74 & 9.54e-02 & 9.56e-02 $\pm$ 1.30e-03 \\ 
	   \bottomrule
    \end{tabular}
  \end{table}
  \begin{table}[ht]
    \centering
    \footnotesize
      \caption{Approximation of the PDF of a lognormal sum with $n=64$ and $\sigma=0.072$.} 
            \label{TablePDF5}
    \begin{tabular}{ccccccc}
      \toprule
       {$x$} & {$nx$} & {$\widetilde{\theta}(x)$} & {Saddle/MC} & {MCIS/MC} \\ 
      \midrule
        0.9703 &    62.10 &     6.51 & 9.99e-04 & 9.84e-04 $\pm$ NaN \\ 
        0.9734 &    62.30 &     5.85 & 3.39e-03 & 3.39e-03 $\pm$ NaN \\ 
        0.9766 &    62.50 &     5.20 & 1.01e-02 & 1.01e-02 $\pm$ 1.36e-04 \\ 
        0.9797 &    62.70 &     4.55 & 2.65e-02 & 2.68e-02 $\pm$ 3.59e-04 \\ 
        0.9828 &    62.90 &     3.91 & 6.19e-02 & 6.17e-02 $\pm$ 8.29e-04 \\ 
        0.9859 &    63.10 &     3.28 & 1.26e-01 & 1.26e-01 $\pm$ 1.69e-03 \\ 
	   \bottomrule
    \end{tabular}
  \end{table}


\bibliographystyle{abbrv}
\bibliography{LaplaceLognormalBib,StandardBib}

\appendix

\section{Appendix: R.V. generation}

Let us consider the A-R algorithm \ref{Alg.Exp.Log} for generating from
the exponential family generated by the lognormal.  Here the rejection 
probability is given by $\Exp[\e^{-\log^2(Z)/2\sigma^2}]$ where 
$Z\sim\mathrm{Gamma}(\alpha+1,\alpha)$ and $\alpha=\LW(\theta\sigma^2\e^{-\sigma^2})/\sigma^2$.  
After some manipulations we can rewrite the acceptance probability as
\begin{align*}
 \Exp[\e^{-\log^2(Z)/2\sigma^2}]
   &=\int_0^\infty\e^{-\log^2z/2\sigma^2}
   \frac{\alpha^{\alpha+1}z^{\alpha}\e^{-\alpha z}}{\Gamma(\alpha+1)}\,\dd z\\
   &=\frac{\alpha^{\alpha+1}}{\Gamma(\alpha+1)}\,\sqrt{2\pi}\sigma\,
   \int_0^\infty\frac1{\sqrt{2\pi}\sigma z}
    \e^{-{(\log z-\sigma^2(\alpha+1))^2}/{2\sigma^2}+\sigma^2(\alpha+1)^2/2-\alpha z}\,\dd z\\
   &=\frac{\alpha^{\alpha+1}}{\Gamma(\alpha+1)}\,\sqrt{2\pi}\sigma\,
   \exp\bigg\{\frac{\sigma^2(\alpha+1)^2}{2}\bigg\}\,
   \mathcal{L}\bigg(\alpha\e^{\sigma^2(\alpha+1)}\bigg)
\end{align*}
where $\mathcal{L}(\cdot)$ is the Laplace transform of the lognormal distribution.
From the previous expression it can be observed that when $\theta\rightarrow0$
then $\alpha\rightarrow0$ and also the probability
of acceptance goes to $0$.

In contrast, notice that if $\theta\rightarrow0$ then the probability of acceptance
of algorithm \ref{Alg.Exp.Log.Log} goes to 1.
Thus it seems natural to combine algorithms \ref{Alg.Exp.Log.Log} and \ref{Alg.Exp.Log} into a single one.
We have expressed the probabilities of acceptance as functions of the Laplace transform
of the lognormal distribution so we can insert our approximations, compare their values
an choose the algorithm with the largest approximation of the probability of acceptance.

\section{Appendix: On the estimation of $\mathcal{L}^n(\theta)$.}\label{S:A:Ln}

We first recall the estimator
of $\mathcal{L}(\theta)$ suggested in \cite{AJRN2014}:
\begin{equation}\label{ISestimator.a}
  \widehat{\mathcal{L}}(\theta)=\widetilde{\mathcal{L}}(\theta)  
  \,\vartheta(Y,\theta),  
\end{equation}
where
\begin{equation*}
\widetilde{\mathcal{L}}(\theta)=
  \exp\biggl\{-\frac{\LW^2(\theta\sigma^2)+2\,\LW(\theta\sigma^2)}{2\sigma^2}\biggr\},
 \qquad
   \vartheta(t;\theta)=
  \exp\biggl\{-\frac{\LW(\theta\sigma^2)}{\sigma^2}
  \bigl(\e^{t}-1-t\bigr)\biggr\},
\end{equation*}
and $Y\sim\mathrm{N}(0,\sigma^2)$.
This was proved there to be unbiased  and logarithmically efficient.

\begin{Remark}\label{Rem26.2a}
If $\sigma^2$ is small or modest (as in our numerical examples), $\e^Y-1-Y$
is not too far from $Y^2$. This suggests using 
\begin{equation}\label{26.2b}\widetilde{\mathcal{L}}(\theta) 
 \exp\biggl\{-\frac{\LW(\theta\sigma^2)}{\sigma^2}
  Y^2\biggr\}
\end{equation}
as control variate (\cite[V.2]{AsmussenGlynn2007}) for $\widehat{\mathcal{L}}(\theta)$.
In fact, the requirement for this method that the mean is known is satisfied
because the expectation of \eqref{26.2b} comes out as 
$\bigl(\LW(\theta\sigma^2)+1\bigr)^{1/2}$. We have not implemented this
control variate idea.
\end{Remark}

Now consider the problem of estimating $\mathcal{L}^n(\theta)$. 
Let $\widehat\ell_R$ be \eqref{ISestimator.a} averaged over $R$ replications.
so that $\widehat\ell_R-\mathcal{L}(\theta)$ is approximately normal$(0,\tau^2/R)$
for some $\tau^2$ that can be estimated by the empirical variance $t^2$.
For any smooth function $\varphi$, we can write
\[\varphi(\widehat\ell_R)\ =\ \varphi\bigl(\mathcal{L}(\theta)\bigr)+
\varphi'\bigl(\mathcal{L}(\theta)\bigr)\bigl(\widehat\ell_R-\mathcal{L}(\theta)\bigr)
+\frac{\varphi''\bigl(\mathcal{L}(\theta)\bigr)}{2}
\bigl(\widehat\ell_R-\mathcal{L}(\theta)\bigr)^2
+\Oh(R^{-3/2})\,.
\]
Taking $\varphi(v)=v^n$, this leads to two observations:\\[1mm]
a) using the estimator
\[\widehat\ell_R-\frac{n(n-1)\widehat\ell_R^{n-2}}{2R}t^2\]
instead of $\widehat\ell_R^n$ reduces the bias from
order $R^{-1}$ to order $R^{-3/2}$;\\[1mm]
b) the variance of the (biased) estimator $\widehat\ell_R^n$ 
can be estimated by $n^2\widehat\ell_R^{2n-2}t^2/R$\\[1mm]
[these arguments are essentially the delta method,
\cite[III.3]{AsmussenGlynn2007}].

An alternative to the estimator $\widehat\ell^n$ is to take the average
over $R$ replications over the product of $n$ independent copies
  of $\widehat{\mathcal{L}}(\theta)$.  Clearly this estimator is unbiased
  but of course more costly to produce. The variance can just be estimated
  by standard Monte Carlo.

\section{Appendix: Proof of Proposition \ref{Th5}}\label{S:A:ProofTh5}

We will need the following Lemma:
\begin{Lemma}\label{L845}
 Let 
  \begin{equation}\label{Unbiased}
    \widetilde\beta_n(u):=\e^{\widetilde\theta(u) S_n}\big[{\mathcal{L}}(\widetilde\theta(u))\big]^n\,
    \Ind\{S_n<nu\}.
  \end{equation}
 Then
\begin{equation*}
 \limsup_{u\rightarrow0}\frac{\mathrm{MSE}_{\alpha_n(u)}[\widehat\alpha_n(u)]}
  {\Exp_{\widetilde\theta(u)}\big[\widehat\beta^2_n(u)\big]\,(1+\gamma(u))^{n/2}}\le2.
\end{equation*}
\end{Lemma}
\begin{proof}[Proof of Lemma \ref{L845}]
Note that $\widehat\beta_n(u)$ is unbiased.  This follows from the change of
measure argument
\begin{align*}
 \Exp_{\widetilde\theta(u)}[\widehat \beta_n(u)]
 &=\Exp_{\widetilde\theta(u)}\big[\,\e^{\widetilde\theta(u) S_n+n\kappa(\widetilde\theta_u)}\,;S_n<nu\big]=\Exp[\Ind(S_n<nu)]=\Prob(S_n<nu).
\end{align*}
Hence we can write
\begin{align*}
 \limsup_{u\rightarrow0}\frac{\mathrm{MSE}_{\alpha_n(u)}[\widehat\alpha_n(u)]}
{\Exp_{\widetilde\theta(u)}[\widehat\beta^2_n(u)](1+\gamma(u))^{n/2}}
 &=\limsup_{u\rightarrow0}\frac{\Exp_{\widetilde\theta(u)}\big[\big(\widehat\alpha_n(u)-
  \Exp_{\widetilde\theta(u)}[\widehat\beta_n(u)]\big)^2\big]}
{\Exp_{\widetilde\theta(u)}[\widehat\beta^2_n(u)](1+\gamma(u))^{n/2}}.
\end{align*}
Since all terms are positive then the last expression can be bounded with
\begin{align*}
 \limsup_{u\rightarrow0}\frac{\Exp_{\widetilde\theta(u)}\big[\widehat\alpha_n^2(u)\big]+
  \Exp^2_{\widetilde\theta(u)}\big[\widehat\beta_n(u)]\big]}
{\Exp_{\widetilde\theta(u)}[\widehat\beta^2_n(u)](1+\gamma(u))^{n/2}}
< \limsup_{u\rightarrow0}\frac{\Exp_{\widetilde\theta(u)}\big[\widehat\alpha_n^2(u)\big]+
  \Exp_{\widetilde\theta(u)}\big[\widehat\beta_n^2(u)]\big]}
{\Exp_{\widetilde\theta(u)}[\widehat\beta^2_n(u)](1+\gamma(u))^{n/2}}.
\end{align*}
Using the definition for both $\widehat\alpha_n(u)$ and $\widehat\beta_n(u)$ we arrive at
\begin{equation*}
 \limsup_{u\rightarrow0}\frac{\Exp_{\widetilde\theta(u)}\big[\e^{-2\widetilde\theta(u)S_n};\!S_n<nu]
   \Exp_{\widetilde\theta(u)}[\widehat{\mathcal{L}}^{2n}_{IS}(\widetilde\theta(u))]\!+\!
   \Exp_{\widetilde\theta(u)}\big[\e^{-2\widetilde\theta(u)S_n};\!S_n<nu]\mathcal{L}^{2n}(\widetilde\theta(u))}
  {\Exp_{\widetilde\theta(u)}\big[\e^{-2\widetilde\theta(u)S_n};S_n<nu]\,\mathcal{L}^{2n}(\widetilde\theta(u))\,
   (1+\gamma(u))^{n/2}}.
\end{equation*}
The last simplifies as
\begin{equation*}
 \limsup_{u\rightarrow0}\frac{
   \Exp_{\widetilde\theta(u)}[\widehat{\mathcal{L}}^{2n}_{IS}(\widetilde\theta(u))]+
   \mathcal{L}^{2n}(\widetilde\theta(u))}
  {\mathcal{L}^{2n}(\widetilde\theta(u))\,
   (1+\gamma(u))^{n/2}}.
\end{equation*}
Now, since $\widehat{\mathcal{L}}_{\mathrm{IS}}(\cdot)$ is unbiased, then the following upper
bound is obtained by a direct application of Jensen's inequality
\begin{align*} 
 \limsup_{u\rightarrow0}\frac{2\,
   \Exp_{\widetilde\theta(u)}[\widehat{\mathcal{L}}^{2n}_{IS}(\widetilde\theta(u))]}
  {\mathcal{L}^{2n}(\widetilde\theta(u))(1+\gamma(u))^{n/2}}.
\end{align*}
By using the formula \eqref{ISestimator} and recalling that $\LW(\widetilde\theta(u)\sigma^2)=\gamma(u)$
we arrive to the equivalent expression
\begin{equation*}
    \limsup_{u\rightarrow\infty}\frac{2\,\widetilde{\mathcal{L}}^{2n}(\widetilde\theta(u))\Exp[\vartheta_k^{2n}(Z;\widetilde\theta(u))]}{\widetilde{\mathcal{L}}^{2n}(\widetilde\theta(u))}\le 2.
\end{equation*}
The last inequality follows from the fact that $\vartheta(\cdot)\le1$.
\end{proof}

\begin{proof}[Proof of Proposition \ref{Th5}]
By Lemma \ref{L845} we have that
\begin{equation*}
 \limsup_{u\rightarrow0}\frac{\mathrm{MSE}_{\alpha_n(u)}[\widehat\alpha_n(u)]}{\alpha^{2-\epsilon}_n(u)}\le
  \limsup_{u\rightarrow0}\frac{2\,\Exp_{\widetilde\theta(u)}[\widehat\beta^2_n(u)](1+\gamma(u))^{n/2}}{\alpha^{2-\epsilon}_n(u)}.
\end{equation*}
Using the change of measure argument we can rewrite the expectation in the numerator in the expression above as
\begin{align}
 \Exp_{\theta(u})[\widehat \beta^2_n(u)]
 &=\Exp_{\widetilde\theta(u)}\Big[\,\e^{2\widetilde\theta(u)S_n+2n\kappa(\widetilde\theta(u))}\,;S_n<nu\Big]\nonumber\\
 &=\Exp\Big[{\e^{\widetilde\theta(u) S_n+n\kappa(\widetilde\theta(u))}};S_n<nu\Big].\label{eff2}
\end{align}
Now, observe that $\widetilde\theta(u)S_n<\widetilde\theta(u)nu$ in the set 
$\{S_n<nu\}$, so we obtain the bound
\begin{equation*}
 \e^{\widetilde\theta(u) nu+n\kappa(\widetilde\theta(u))}\,\Exp[\Ind(S_n<nu)]
 ={\e^{\widetilde\theta(u) nu-n\kappa(\widetilde\theta(u))}}\Prob(S_n<nu).
\end{equation*}
Also we have the following inequalities
\begin{equation*}
 \alpha(u)=\Prob(S_n<nu)\ge\Prob(\max\{X_i\}<u)=\Prob^n(Y_1<u).
\end{equation*}
Putting together these results we conclude that
\begin{align*}
 \lim_{u\rightarrow0}\frac{\Exp_\theta[\widehat\beta^2(u)](1+\gamma(u))^{n/2}}{\alpha^{2-\epsilon}(u)}
 &<\lim_{u\rightarrow0}\frac{\,\e^{\widetilde\theta(u) nu+n\kappa(\widetilde\theta(u))}(1+\gamma(u))^{n/2}}{\alpha^{1-\epsilon}(S_n<nu)}\\
 &<\lim_{u\rightarrow0}\bigg[\frac{\e^{\widetilde\theta(u) u+\kappa(\widetilde\theta(u))}\sqrt{1+\gamma(u)}}{\Prob^{1-\epsilon}(Y_1<u)}\bigg]^n\\
 &=\bigg[\lim_{u\rightarrow0}\frac{\e^{\widetilde\theta(u) u+\kappa(\widetilde\theta(u))}\sqrt{1+\gamma(u)}}{\Prob^{1-\epsilon}(Y_1<u)}\bigg]^n,
\end{align*}
where the last equality holds provided that the limit inside the brackets exists.
In fact, we will prove that such limit is $0$.
For the numerator in the limit above we use \eqref{10.2a} to obtain
\begin{equation*}
 \e^{\widetilde\theta(u) u+\kappa(\widetilde\theta(u))}\sqrt{1+\gamma(u)}\sim 
 \exp\bigg\{\widetilde\theta(u)u-\frac{\gamma^2(u)+2\gamma(u)}{2\sigma^2} \bigg\}\sqrt{1+\gamma(u)},
\end{equation*}
while for the denominator we employ \emph{Mill's ratio} so we get
\begin{equation*}
 \Prob(Y_1<u)\sim\frac\sigma{\sqrt{2\pi}|\log u|}\exp
 \bigg\{-\frac{\log^2 u}{2\sigma^2}\bigg\},\qquad u\rightarrow0.
\end{equation*}
Hence
\begin{align*}
 &\lim_{u\rightarrow0}\frac{\e^{\widetilde\theta(u) u+\kappa(u)}}{\Prob^{1-\epsilon}(Y_1<u)}\\
 &\qquad=\lim_{u\rightarrow0}\frac{c_1\,\sqrt{1+\gamma(u)}}{|\log u|^{-(1-\epsilon)}}
 \exp\bigg\{-\frac{\gamma^2(u)+2\gamma(u)}{2\sigma^2} +u\widetilde\theta(u)+(1-\epsilon)
 \frac{\log^2 u}{2\sigma^2}\bigg\}\\
 &\qquad=\lim_{u\rightarrow0}\frac{c_1\,\sqrt{1+\gamma(u)}}{|\log u|^{-(1-\epsilon)}}
 \exp\bigg\{-\frac{\gamma(u)}{\sigma^2}+u\widetilde\theta(u)-
 \frac{\gamma^2(u)-(1-\epsilon)\log^2 u}{2\sigma^2}\bigg\},
\end{align*}
for some constant $c_1$.  We apply 
Lemmas \ref{limit2}--\ref{limit3} to prove that the last limit is
\begin{equation*}
  \lim_{u\rightarrow0}{c_2\,|\log u|^{3/2-\epsilon}}
 \exp\bigg\{-\epsilon\frac{\log^2 u}{2\sigma^2}+\Oh(|\log u|^{-2})\bigg\}.
\end{equation*}
The last limit is equal to $0$ for all $\epsilon>0$. This completes the proof.
\end{proof}
\begin{Lemma}\label{limit1}
 \begin{equation*}
  \gamma(u)=|\log(u)|+\frac{\sigma^2}{|\log u|}+\Oh(|\log u|^{-2}), \qquad u\rightarrow0.
 \end{equation*}
\end{Lemma}
\begin{proof}
 \begin{align*}
  \gamma(u)&=\dfrac{-1+|\log u|+\big(1+|\log u|\big)\sqrt{1+\dfrac{2\sigma^2}{(1+|\log u|)^2}}}{2}\\
   &=\dfrac{-1+|\log u|+\big(1+|\log u|\big)\big(1+\frac{\sigma^2}{\log^2 u}+\Oh(\log^{-3}u)\big)}{2}\\
   &=|\log u|+\frac{\sigma^2}{|\log u|}+\Oh(|\log u|^{-2}).
 \end{align*}
\end{proof}
\begin{Lemma}\label{limit2}
\begin{equation*}
  {u\widetilde\theta(u)}=\frac{|\log u|}{\sigma^2}+1+\Oh(|\log u|^{-1}).
\end{equation*}
\end{Lemma}
\begin{proof}
From Lemma \ref{limit1} it follows that for $u>0$ close enough to $0$ it holds that
\begin{equation*}
  u\e^{\gamma(u)}=\exp\bigg\{\dfrac{\sigma^2}{|\log u|}+\Oh(|\log u|^{-2})\bigg\}
  =1+\dfrac{\sigma^2}{|\log u|}+\Oh(|\log u|^{-2}).
\end{equation*}
Using this result and combining again with Lemma \ref{limit1} we get
\begin{equation*}
 {u\widetilde\theta(u)}=\frac{u\gamma(u)\e^{\gamma(u)}}{\sigma^2}
  =\frac{|\log u|}{\sigma^2}+1+\Oh(|\log u|^{-1})
\end{equation*}
\end{proof}

\begin{Lemma}\label{limit3}
\begin{equation*}
  \frac{\gamma^2(u)-\log^2 u}{2\sigma^2}=1+\Oh(|\log u|^{-2}).
\end{equation*}
\end{Lemma}
\begin{proof}[Proof of Lemma \ref{limit3}]
From Lemma \ref{limit1} it follows that
  \begin{equation*}
  \gamma^2(u)
  -|\log u|^2=2\sigma^2+\Oh(|\log u|^{-2})
 \end{equation*}
\end{proof}

\end{document}